\documentclass{article}%
\usepackage{amsmath}
\usepackage{amsfonts}
\usepackage{amssymb}
\usepackage{graphicx}%
\providecommand{\U}[1]{\protect\rule{.1in}{.1in}}
\newtheorem{theorem}{Theorem}

\newtheorem{corollary}[theorem]{Corollary}

\newtheorem{lemma}[theorem]{Lemma}

\newtheorem{proposition}[theorem]{Proposition}
\newtheorem{remark}[theorem]{Remark}

\newenvironment{proof}[1][Proof]{\noindent\textbf{#1.} }{\ \rule{0.5em}{0.5em}}
\begin{document}

\title{The Correlation Numerical Range of a Matrix and Connes' Embedding
Problem\thanks{Email: \texttt{don@math.unh.edu} and
\texttt{deguang.han@ucf.edu}. This Research is supported in part by an NSF
grant}}
\author{Don Hadwin\\Department of Mathematics and Statistics\\University of New Hampahire\\Durham, NH 03824, USA
\and Deguang Han\\Department of Mathematics\\University of Central Florida\\Orlando, FL 32816, USA}
\maketitle

\begin{abstract}
We define a new numerical range of an $n\times n$ complex matrix in terms of
correlation matrices and develop some of its properties. We also define a
related numerical range that arises from Alain Connes' famous embedding problem.

\vspace{1ex}



\end{abstract}


\section{Introduction and Preliminaries}

In this paper we define and study a new numerical range for $n\times n$
complex matrices. This numerical range is loosely related to new
reformulations \cite{JP}, \cite{KS} of Connes' Embedding problem (CEP)
\cite{C}. We derive some of the basic properties of this new range, pose some
questions. We answer all of the questions in the $2\times2$ case. We introduce
another related numerical range with many of the same properties that is
directly related to the reformulation of CEP in \cite{JP} and further
developed in \cite{DJ}.

\bigskip

If $T$ is an operator on a Hilbert space $\mathcal{H}$, the \emph{numerical
range} $W\left(  T\right)  $ of $T$ is the set%
\[
W\left(  T\right)  =\left\{  \left\langle Tx,x\right\rangle :x\in
\mathcal{H},\left\Vert x\right\Vert =1\right\}  .
\]
Suppose $A=\left(  \alpha_{ij}\right)  $ is an $n\times n$ complex matrix. We
define $\hat{A}=\left(  \alpha_{ijI_{n}}\right)  $ acting on $H=\mathbb{C}%
^{n}\mathbb{\oplus\cdots\oplus C}^{n}$ ($n$ copies), and let
\[
E=\left\{  \frac{1}{\sqrt{n}}\left(
\begin{array}
[c]{c}%
e_{1}\\
e_{2}\\
\vdots\\
e_{n}%
\end{array}
\right)  \in H:\left\Vert e_{1}\right\Vert =\cdots=\left\Vert e_{n}\right\Vert
=1\right\}
\]
$.$ We define the \emph{correlation numerical range}, or $C$\emph{-numerical
range}, of $A$ as%
\[
W_{c}\left(  A\right)  =\left\{  \left\langle \hat{A}e,e\right\rangle :e\in
E\right\}  .
\]

The term \emph{correlation} comes from the fact that an $n\times n$
\emph{correlation matrix} \cite{LS} is a matrix $B=\left(  b_{ij}\right)
\in\mathcal{M}_{n}\left(  \mathbb{C}\right)  $ such that $B\geq0$ and
$b_{ii}=1$ for $1\leq i\leq n$. Equivalently, $B$ is an $n\times n$
correlation matrix if there are unit vectors $e_{1},\ldots,e_{n}\in
\mathbb{C}^{n}$ such that%
\[
B=\left(  \left\langle e_{i},e_{j}\right\rangle \right)  .
\]
Let $\mathcal{E}_{n}$ denote the set of $n\times n$ correlation matrices. If
$A\in\mathcal{M}_{n}\left(  C\right)  $, then $A^{T}$ denotes the transpose of
$T$. We also use $\mathcal{D}_{n}$ to denote the set of all the $n\times n$
diagonal matrices, and $\mathcal{D}_{n, 0}$ the set of trace zero $n\times n$
diagonal matrices.

\section{Basic Results}

We first prove some of the basic properties of $W_{c}\left(  T\right)  $.

\begin{theorem}
\label{basic} Suppose $A\in\mathcal{M}_{n}\left(  \mathbb{C}\right)  $. Then

\begin{enumerate}
\item $W_{c}\left(  A\right)  =\left\{  \tau_{n}(AB\right\}  :B=\left\{
\left(  b_{ij}\right)  \in\mathcal{M}_{n}\left(  \mathbb{C}\right)
,B\geq0,b_{ii}=1\text{ for }1\leq i\leq n\right\}  $.

\item $W_{c}\left(  A\right)  \subseteq W\left(  A\right)  $

\item $W_{c}\left(  A\right)  $ is convex

\item $\tau_{n}\left(  A\right)  \in W_{c}\left(  A\right)  $

\item If $D\in\mathcal{D}_{n}$, then
\[
W_{c}\left(  A+D\right)  =W_{c}\left(  A\right)  +\tau_{n}\left(  D\right)
\]

\item $W_{c}\left(  A\right)  =\left\{  \lambda\right\}  $ if and only if $A$
is diagonal and $\tau_{n}\left(  A\right)  =\lambda$

\item $W_{c}\left(  A\right)  \subseteq\mathbb{R}$ if and only if
$\operatorname{Im}A$ is diagonal and $\tau_{n}\left(  \operatorname{Im}%
A\right)  =0.$

\item If $\left\{  A_{k}\right\}  $ is a sequence in $\mathcal{M}_{n}\left(
\mathbb{C}\right)  $ and $\left\Vert A_{k}-A\right\Vert \rightarrow0$, then%
\[
W_{c}\left(  A\right)  =\left\{  \lim_{k\rightarrow\infty}\lambda_{k}%
:\lambda_{k}\in W_{c}\left(  A_{k}\right)  \text{ for }k\geq1\text{ and }%
\lim_{k\rightarrow\infty}\lambda_{k}\text{ exists}\right\}  .
\]

\item $W_{c}\left(  A^{T}\right)  =W_{c}\left(  A\right)  .$
\end{enumerate}
\end{theorem}

\begin{proof}
$\left(  1\right)  $. This is a direct calculation,

$\left(  2\right)  $. This follows from the definition of $W_{c}\left(
A\right)  $ and the fact that $W\left(  A\right)  =W\left(  \hat{A}\right)  $

$\left(  3\right)  .$ This follows from $\left(  1\right)  $.

$\left(  4\right)  $. Choose a vector in $E$ with $\left\{  e_{1},\ldots
,e_{n}\right\}  $ an orthonormal basis for $\mathbb{C}^{n}$.

$\left(  5\right)  $ This is a direct computation.

$\left(  6\right)  $. The "if" part follows from $\left(  5\right)  $. For the
other direction, suppose $W_{c}\left(  A\right)  =\left\{  \lambda\right\}  $.
We know from $\left(  4\right)  $ that $\lambda=\tau_{n}\left(  A\right)  $.
Suppose $i\neq j$. Choose unit vectors $\left\{  e_{1},\ldots,e_{n}\right\}  $
in $\mathbb{C}^{n}$ so that $\left\{  e_{k}:i\neq k\neq j\right\}  $ is
orthonormal and orthogonal to $\left\{  e_{i},e_{j}\right\}  $ and such that
$e_{j}=\beta$. We have that
\[
\lambda=\left\langle \hat{A}e,e\right\rangle =\tau_{n}\left(  A\right)
+\alpha_{ij}\bar{\beta}+\alpha_{ji}\beta.
\]
Since $\beta\in\mathbb{C}$ with $\left\vert \beta\right\vert =1$ it follows
that $\alpha_{ij}=0$.

$\left(  7\right)  $. This follows from $\left(  6\right)  $ and the obvious
fact that%
\[
\operatorname{Re}W_{c}\left(  A\right)  =W_{c}\left(  \operatorname{Re}%
A\right)  \text{ and }\operatorname{Im}W_{c}\left(  A\right)  =W_{c}\left(
\operatorname{Im}A\right)  .
\]

$\left(  8\right)  $. This is an easy consequence of the compactness of the
set $\mathcal{E}_{n}$ of $n\times n$ correlation matrices.

$\left(  9\right)  $. This follows from $\left(  1\right)  $, the fact that
$\mathcal{E}_{n}=\left\{  B^{T}:B\in\mathcal{E}_{n}\right\}  ,$ and the fact
that $\tau_{n}\left(  S^{T}\right)  =\tau_{n}\left(  S\right)  $ for every
$S\in\mathcal{M}_{n}\left(  \mathbb{C}\right)  $.
\end{proof}

\bigskip

Here is a fundamental problem in this paper.

\bigskip

\textbf{Problem 1:} What is a necessary and sufficient condition for
$W_{c}\left(  A\right)  \subseteq\lbrack0,\infty)$ or $\left(  0,\infty
\right)  $? In particular, is it true that $W_{c}\left(  A\right)
\subseteq\lbrack0,\infty)$ if an only if $A$ is the sum of a trace-zero
diagonal operator and a positive semidefinite operator?

\bigskip

Note that if $A$ is a limit of matrices of the form "positive semidefinite +
zero-trace diagonal", then $W_{c}\left(  A\right)  \subseteq\lbrack0.\infty)$.
However, this set of matrices is norm closed.

\bigskip

\begin{lemma}
Suppose $A\in\mathcal{M}_{n}\left(  \mathbb{C}\right)  $. The following are equivalent.

\begin{enumerate}
\item $A$ is the sum of a positive semidefinite matrix and a trace-zero
diagonal matrix

\item There is a sequence $\left\{  A_{k}\right\}  $ of positive semidefinite
operators and a sequence $\left\{  D_{k}\right\}  $ of trace-zero diagonal
operators such that $\left\Vert A_{k}+D_{k}-A\right\Vert \rightarrow0$.
\end{enumerate}
\end{lemma}

\begin{proof}
The implication $\left(  1\right)  \Longrightarrow\left(  2\right)  $ is
obvious. Suppose $\left(  2\right)  $ is true. Then
\[
Tr\left(  A_{k}\right)  =Tr\left(  A_{k}+D_{k}\right)  \rightarrow Tr(A).
\]
Since $0\leq A_{k}$, we have $\left\Vert A_{k}\right\Vert \leq Tr(A_{k})$.
Hence there is a subsequence $\left\{  A_{k_{j}}\right\}  $ that converges to
some $P\geq0$. Also%
\[
D_{k_{j}}=\left(  A_{k_{j}}+D_{k_{j}}\right)  -A_{k_{j}}\rightarrow A-P=D
\]
for some zero-trace diagonal operator $D$. Hence $A=P+D$ shows that $\left(
1\right)  $ is true.
\end{proof}

\bigskip

It is easily shown that $W\left(  S\oplus T\right)  $ is the convex hull of
$W\left(  S\right)  \cup W\left(  T\right)  $. Here is the analogue for
$W_{c}$.

\bigskip

\begin{lemma}
\label{ds}Suppose $A\in\mathcal{M}_{n}\left(  \mathbb{C}\right)  $ and
$A=\left(
\begin{array}
[c]{cc}%
S_{1} & 0\\
0 & S_{2}%
\end{array}
\right)  $ where $S_{j}\in\mathcal{M}_{k_{j}}\left(  \mathbb{C}\right)  $ for
$k=1,2$. Then%
\[
W_{c}\left(  A\right)  =\frac{k_{1}}{n}W_{c}\left(  S_{1}\right)  +
\frac{k_{2}}{n}W_{c}\left(  S_{2}\right)  .
\]
\bigskip
\end{lemma}

It follows from Theorem \ref{basic} that in $W_{c}$ is actually a function on
$\mathcal{M}_{n}\left(  \mathbb{C}\right)  /\mathcal{D}_{n,0}$, i.e., in
comparing $W_{c}\left(  T\right)  $ with $W\left(  T\right)  $ we see that
$\mathcal{D}_{n}$ plays the role of $\mathbb{C}I_{n}$ and $\mathcal{D}_{n,0}$
plays the role of $0$. Here are some more examples.

It is true that $W\left(  U^{\ast}TU\right)  =W\left(  T\right)  $ for every
operator $T$ and every unitary operator $U$. It is known that, for every
$\ast$-automorphism $\alpha$ of $\mathcal{M}_{n}\left(  \mathbb{C}\right)  ,$
there is a unitary matrix $U$ such that $\alpha=ad_{U}$, i.e.,%
\[
\alpha\left(  T\right)  =U^{\ast}TU
\]
for every $T\in\mathcal{M}_{n}\left(  \mathbb{C}\right)  $. Let $\mathcal{G}%
_{n}$ be the group of unitary matrices generated by the diagonal unitaries and
the permutation matrices.

\bigskip

\begin{proposition}
\label{normalizer}Suppose $U\in\mathcal{M}_{n}\left(  \mathbb{C}\right)  $
unitary. The following are equivalent.

\begin{enumerate}
\item $ad_{U}\left(  \mathcal{D}_{n}\right)  \subseteq\mathcal{D}_{n}$

\item $ad_{U}\left(  \mathcal{D}_{n,0}\right)  \subseteq\mathcal{D}_{n,0}$

\item $W_{c}\left(  U^{\ast}AU\right)  =W_{c}\left(  A\right)  $ for every
$A\in\mathcal{M}_{n}\left(  \mathbb{C}\right)  $

\item $U\in\mathcal{G}_{n}$
\end{enumerate}
\end{proposition}

\begin{proof}
$\left(  1\right)  \Leftrightarrow\left(  2\right)  $. This is obvious since
$ad_{U}$ always preserve the trace and $ad_{U}\left(  I\right)  =I$

$\left(  3\right)  \Longrightarrow\left(  1\right)  $. This follows from the
fact that $\mathcal{D}_{n}$ is the set of all $T\in\mathcal{M}_{n}$ such that
$W_{c}\left(  T\right)  $ is a singleton.

$\left(  1\right)  \Longrightarrow\left(  4\right)  $. This is well-known.
Since $ad_{U}$ is a automorphism of $\mathcal{D}_{n}$ and $\mathcal{D}_{n}$ is
$\ast$-isomorphic to $C\left(  \left\{  1,2,\ldots,n\right\}  \right)  $, and
since every automorphism on $C\left(  \left\{  1,2,\ldots,n\right\}  \right)
$ is composition with a homeomorphism on $\left\{  1,2,\ldots,n\right\}  $,
there is a unitary permutation matrix $W\in\mathcal{M}_{n}\left(
\mathbb{C}\right)  $ such that
\[
ad_{U}\left(  D\right)  =ad_{W}\left(  D\right)
\]
for every $D\in\mathcal{D}_{n}$. Hence $UW^{\ast}$ commutes with every
diagonal matrix, i.e., $V=UW^{\ast}$ is a diagonal unitary matrix, so
$U=VW\in\mathcal{G}_{n}$.

$\left(  4\right)  \Longrightarrow\left(  3\right)  .$ It is easily seen that
if $U\in\mathcal{G}_{n}$ and $B\in\mathcal{M}_{n}\left(  \mathbb{C}\right)  $,
then $B$ is a correlation matrix if and only if $UBU^{\ast}$ is a correlation
matrix. Moreover,
\[
\tau_{n}\left(  ad_{U}\left(  A\right)  B\right)  =\tau_{n}\left(  AUBU^{\ast
}\right)  ,
\]
so it follows from part $\left(  1\right)  $ of theorem \ref{basic} that
$W_{c}\left(  ad_{U}\left(  A\right)  \right)  =W_{c}\left(  A\right)  $
whenever $U\in\mathcal{G}_{n}$.
\end{proof}

\bigskip

\bigskip

\section{The Correlation Numerical Radius}

The numerical radius $w\left(  T\right)  $ is defined by%
\[
w\left(  T\right)  =\sup\left\{  \left\vert \lambda\right\vert :\lambda\in
W\left(  T\right)  \right\}  .
\]
\bigskip

We define the $C$-numerical radius as
\[
w_{c}\left(  A\right)  =\max\left\{  \left\vert \lambda\right\vert :\lambda\in
W_{c}\left(  A\right)  \right\}  .
\]
\bigskip

A classical result is that
\[
\left\Vert T\right\Vert /2\leq w\left(  T\right)  \leq\left\Vert T\right\Vert
.
\]
Define
\[
\left\Vert T\right\Vert _{c}=\inf_{D\in\mathcal{D}_{n,0}}\left\Vert
T-D\right\Vert ,
\]
which is the norm of the image of $T$ in $\mathcal{M}_{n}\left(
\mathbb{C}\right)  /\mathcal{D}_{n,0}$.

\begin{proposition}
Suppose $n\in\mathbb{N}$. There is a number $\kappa_{n}>0$ such that
\[
\kappa_{n}\left\Vert T\right\Vert _{c}\leq w_{c}\left(  T\right)
\leq\left\Vert T\right\Vert _{c}%
\]
for every $T\in\mathcal{M}_{n}\left(  \mathbb{C}\right)  $. Moreover, when
$n\geq2,$%
\[
\frac{1}{4n+2}\leq\kappa_{n}\leq\frac{2}{n}.
\]

\end{proposition}

\begin{proof}
We know that $\left\Vert {}\right\Vert _{c}$ and $w_{c}$ are seminorms on
$\mathcal{M}_{n}\left(  \mathbb{C}\right)  $ that are $0$ exactly on
$\mathcal{D}_{n,0}$. Since $\mathcal{M}_{n}\left(  \mathbb{C}\right)  $ is
finite-dimension, these seminorms are equivalent and the existence of
$\kappa_{n}$ is proved. Moreover, for any $T\in\mathcal{M}_{n}\left(
\mathbb{C}\right)  $ and any $D\in\mathcal{D}_{n,0}$, we have
\[
w_{c}\left(  T\right)  =w_{c}\left(  T-D\right)  \leq\left\Vert T-D\right\Vert
.
\]
Hence $w_{c}\left(  T\right)  \leq\left\Vert T\right\Vert _{c}$.

Suppose $n\geq2$ and let $A\in\mathcal{M}_{n}\left(  \mathbb{C}\right)  $ be
the direct sum of $\left(
\begin{array}
[c]{cc}%
0 & 1\\
0 & 0
\end{array}
\right)  $ with an $\left(  n-2\right)  \times\left(  n-2\right)  $ zero
matrix. It follows from Lemma \ref{ds} that $w_{c}\left(  A\right)  =\frac
{2}{n}$ and $\left\Vert T\right\Vert _{c}=1.$ Thus $\kappa_{n}\leq2/n$.

Suppose $T=T^{\ast}$ is a matrix all of whose diagonal entries are $0.$ It
follows from \cite{R} that there is a diagonal projection matrix $P$ such that%
\[
dist\left(  T,\mathcal{D}_{n}\right)  \leq2\left\Vert TP-PT\right\Vert .
\]
Let $\left\Vert {}\right\Vert _{1}$ denote the trace-class norm on
$\mathcal{M}_{n}\left(  \mathbb{C}\right)  $, i.e., $\left\Vert S\right\Vert
_{1}=Tr\left(  \left(  S^{\ast}S\right)  ^{\frac{1}{2}}\right)  $. Suppose
$X\in\mathcal{M}_{n}\left(  \mathbb{C}\right)  $ and $X=\left(  1-P\right)
XP$ and $\left\Vert X\right\Vert \leq1$. Then $X+X^{\ast}$ is a selfadjoint
zero-diagonal contraction, so $P+X+X^{\ast}=B=\left(  b_{ij}\right)  $ is a
correlation matrix. Moreover, if we write $T=\left(  t_{ij}\right)  $ and
$\left(  1-P\right)  TP=\left(  s_{ij}\right)  $
\[
Tr\left(  TB^{T}\right)  =%
{\displaystyle\sum_{i\neq j}}
t_{ij}b_{is}=2\operatorname{Re}%
{\displaystyle\sum_{{}}}
s_{ij}x_{ij},
\]
so we get%
\[
w_{c}\left(  T\right)  \geq\left\vert \tau_{n}\left(  TB^{T}\right)
\right\vert \geq\frac{2}{n}\left\vert
{\displaystyle\sum_{{}}}
s_{ij}x_{ij}\right\vert .
\]
Taking the supremum over all $X$ we obtain%
\[
w_{c}\left(  T\right)  \geq\frac{2}{n}\left\Vert \left(  1-P\right)
TP\right\Vert _{1}\geq\frac{2}{n}\left\Vert \left(  1-P\right)  TP\right\Vert
\geq\frac{2}{n}dist\left(  T,\mathcal{D}_{n}\right)  .
\]
But if we choose $D\in\mathcal{D}_{n}$ such that $\left\Vert T-D\right\Vert
=dist\left(  T,\mathcal{D}_{n}\right)  ,$ we have that
\[
dist\left(  D,\mathcal{D}_{n,0}\right)  =\left\vert \tau_{n}\left(  D\right)
\right\vert =\left\vert \tau_{n}\left(  T-D\right)  \right\vert \leq\frac
{n}{2}w_{c}\left(  T\right)  .
\]
Hence,
\[
\left\Vert T\right\Vert _{c}\leq\left\Vert T-D\right\Vert +dist\left(
D,\mathcal{D}_{n,0}\right)  \leq nw_{c}\left(  T\right)  .
\]

For an arbitrary $T=T^{\ast}$, we can write $T=T_{1}+D_{1}$ with $T_{1}%
=T_{1}^{\ast}$ a zero-diagonal matrix and $D_{1}\in\mathcal{D}_{n}$. We have%
\[
\left\Vert T\right\Vert _{c}\leq\left\Vert T_{1}\right\Vert _{c}+\left\Vert
D_{1}\right\Vert _{c}=\left\Vert T_{1}\right\Vert _{c}+\left\vert \tau
_{n}\left(  T\right)  \right\vert \leq
\]%
\[
nw_{c}\left(  T_{1}\right)  +w_{c}\left(  T\right)  \leq n\left[  w_{c}\left(
D_{1}\right)  +w_{c}\left(  T\right)  \right]  +w_{c}\left(  T\right)
\leq\left(  2n+1\right)  w_{c}\left(  T\right)  .
\]
For the general case, we know from the fact that $w_{c}\left(  T^{\ast
}\right)  =w_{c}\left(  T\right)  $ and $\left\Vert T^{\ast}\right\Vert
_{c}=\left\Vert T\right\Vert _{c}$ that, for an arbitrary $T\in\mathcal{M}%
_{n}\left(  \mathbb{C}\right)  $ we have
\[
\left\Vert T\right\Vert _{c}\leq\left\Vert \operatorname{Re}T\right\Vert
_{c}+\left\Vert \operatorname{Im}T\right\Vert _{c}\leq\left(  2n+1\right)
\left[  w_{c}\left(  \operatorname{Re}T\right)  +w_{c}\left(
\operatorname{Im}T\right)  \right]  \leq
\]%
\[
\left(  2n+1\right)  \left[  w_{c}\left(  T\right)  +w_{c}\left(  T\right)
\right]  \leq\left(  4n+2\right)  w_{n}\left(  T\right)  .
\]
Hence $\kappa_{n}\geq\frac{1}{4n+2}$.
\end{proof}

\bigskip

In the preceding proposition, there was no attempt to get the best estimates
of $\kappa_{n}$. This leads to a natural question. \vspace{3mm}

\textbf{Problem 2. }What is the exact value of $\kappa_{n}$?

\bigskip

The following result relates $w$ and $w_{c}$.

\bigskip

\begin{proposition}
Suppose $A\in\mathcal{M}_{n}\left(  \mathbb{C}\right)  $. Then%
\[
w_{c}\left(  A\right)  \leq\inf_{D\in\mathcal{D}_{n,0}}w\left(  A+D\right)  .
\]

\end{proposition}

\begin{proof}
We know that $w_{c}\left(  A\right)  =w_{c}\left(  A+D\right)  \leq w\left(
A+D\right)  $ for every $D\in\mathcal{D}_{n,0}$, and the result follows.
\end{proof}

\section{Connes' Embedding Problem; Correlation Matrices From Unitaries}

In this section we consider a special class of correlation matrices. Suppose
$k$ is a positive integer. We can make $\mathcal{M}_{k}\left(  \mathbb{C}%
\right)  $ into a Hilbert space with inner product
\[
\left\langle S,T\right\rangle =\tau_{k}\left(  T^{\ast}S\right)  .
\]
In this case we have that each unitary matrix $U$ in $\mathcal{M}_{k}\left(
\mathbb{C}\right)  $ is a unit vector. We say that an $n\times n$ correlation
matrix is \emph{unitarily induced} if there is a positive integer $k$ and
unitary $k\times k$ matrices $U_{1},\ldots,U_{k}$ such that%
\[
A=\left(  \tau_{n}\left(  U_{j}^{\ast}U_{i}\right)  \right)  .
\]
The set $\mathcal{F}_{n}$ of $n\times n$ unitarily induced correlation
matrices is generally not convex. This is because if $U$ and $V$ are unitary
and $s,t$ are nonzero numbers, then $sU+tV$ is unitary if and only if%
\[
\sigma\left(  s+tU^{\ast}V\right)  \subset\mathbb{T},
\]
where $\mathbb{T}=\left\{  \lambda\in\mathbb{C}:\left\vert \lambda\right\vert
=1\right\}  $; equivalently if
\[
\sigma\left(  U^{\ast}T\right)  \subset\frac{1}{t}\mathbb{T}-\frac{s}{t}.
\]
Details of this are worked out in \cite{DJ}. If $T\in\mathcal{M}_{n}\left(
\mathbb{C}\right)  $, we define%
\begin{align*}
W_{uc}\left(  T\right)   &  =co\left(  \left\{  \tau_{n}\left(  TA\right)
:A\in\mathcal{F}_{n}\right\}  \right)  =\\
&  \left\{  \tau_{n}\left(  TB\right)  :B\in co\mathcal{F}_{n}\right\}  .
\end{align*}
Since $\mathcal{F}_{n}\subset\mathcal{E}_{n}$, it is clear that $W_{uc}\left(
T\right)  \subset W_{c}\left(  T\right)  .$ It was proved in \cite{LS} (later
in \cite{DJ}) that $\mathcal{E}_{n}$ is the convex hull of $\mathcal{F}_{n}$
if and only if $n\leq3$.

\bigskip

\begin{lemma}
Suppose $n$ is a positive integer. Then

\begin{enumerate}
\item If $n\leq3$ and $T\in\mathcal{M}_{n}\left(  \mathbb{C}\right)  $, then%
\[
W_{c}\left(  T\right)  =W_{uc}\left(  T\right)  .
\]

\item If $n>3$, then there is a $T=T^{\ast}$ in $M_{n}\left(  \mathbb{C}%
\right)  $ such that%
\[
W_{uc}\left(  T\right)  \neq W_{c}\left(  T\right)  .
\]

\end{enumerate}
\end{lemma}

\begin{proof}
$\left(  1\right)  .$ This follows from the fact \cite{LS} that $\mathcal{E}%
_{n}=co\mathcal{F}_{n}$ when $n\leq3$.

$\left(  2\right)  $. Suppose $n>3$. It follows from \cite{LS} that there is a
$B\in\mathcal{E}_{n}$ such that $B\notin co\mathcal{F}_{n}$. It follows from
the Hahn Banach theorem that there is a continuous linear functional $\varphi$
on $\mathcal{M}_{n}\left(  \mathbb{C}\right)  $ and a real number $t$ such
that%
\[
\operatorname{Re}\varphi\left(  A\right)  <t<\varphi\left(  B\right)
\]
for every $A\in co\mathcal{F}_{n}$. We know that there is a $K\in
\mathcal{M}_{n}\left(  \mathbb{C}\right)  $ such that%
\[
\varphi\left(  S\right)  =\tau_{n}\left(  SK\right)
\]
for every $S\in\mathcal{M}_{n}\left(  \mathbb{C}\right)  $. If $S\geq0$, then%
\[
\operatorname{Re}\varphi\left(  S\right)  =\operatorname{Re}\tau_{n}\left(
S^{\frac{1}{2}}KS^{\frac{1}{2}}\right)  =\tau_{n}\left(  S\operatorname{Re}%
K\right)  .
\]
Hence, if $T=\operatorname{Re}K$, we have%
\[
W_{uc}\left(  T\right)  \subset(-\infty.t)\text{ and }t<\tau_{n}(TB)\in
W_{c}\left(  T\right)  .
\]
Hence $W_{uc}\left(  T\right)  \neq W_{c}\left(  T\right)  $.
\end{proof}

\bigskip

We now relate $W_{uc}\left(  T\right)  $ to Connes' famous embedding problem
\cite{C}, which asks if every finite von Neumann algebra can be tracially
embedded in an ultraproduct of an ultrapower of the hyperfinite $II_{1}$
factor. Fortunately, the reader does not need to know the meaning of any of
the terms in the preceding sentence, because of a lovely reformulation
\cite{JP} of Connes' embedding problem in terms of matrices. This
reformulation, which is an extension of results in \cite{H} and \cite{KS} ,
was further studied in \cite{DJ}.

Let $\mathbb{F}_{n}$ denote the free group on $n$ generators $\left\{
u_{1},\ldots,u_{n}\right\}  $, and let $\mathcal{A}_{n}$ denote the group
algebra of $\mathbb{F}_{n}$. The definition $u^{\ast}=u^{-1}$ on
$\mathbb{F}_{n}$ extends to an involution $\ast$ on $\mathcal{A}_{n}$.

If $T=\left(  \alpha_{ij}\right)  \in\mathcal{M}_{n}\left(  \mathbb{C}\right)
$, we define an element $p_{T}\left(  u_{1},\ldots,u_{n}\right)
\in\mathcal{A}_{n}$ by%
\[
p_{T}\left(  u_{1},\ldots,u_{n}\right)  =\sum_{i,j=1}^{n}\alpha_{ij}%
u_{i}^{\ast}u_{j}.
\]
Let $\mathcal{P}_{n}$ be the set of elements of $\mathcal{A}_{n}$ that can be
written in the form%
\[
\sum_{j=1}^{m}q_{j}^{\ast}q_{j}+\sum_{j=1}^{m}\left(  f_{j}g_{j}-g_{j}%
f_{j}\right)  .
\]
elements $q_{1},\ldots,q_{m},f_{1},g_{1},\ldots f_{m},g_{m}\in\mathcal{A}_{n}$.

The revised version of Connes' embedding problem in \cite{JP} is equivalent to
the following:

For every positive integer $n$, and every selfadjoint $A$ in $\mathcal{M}%
_{n}\left(  \mathbb{C}\right)  $, if $W_{uc}\left(  A\right)  \subset
(0,\infty)$, then $p_{A}\in\mathcal{P}_{n}$.

\bigskip

This question was answered affirmatively by Popovych \cite{JP} when $n=3.$

\bigskip

Note that $\mathcal{P}_{n}$ is closed under addition and multiplication by
nonnegative scalars and that the set of $A=A^{\ast}$ in $M_{n}\left(
\mathbb{C}\right)  $ such that $W_{uc}\left(  \left(  0,\infty\right)
\right)  $ is also closed under addition and multiplication by nonnegative
scalars. Note that the map $T\rightarrow p_{T}$ from $\mathcal{M}_{n}\left(
\mathbb{C}\right)  $ to $\mathcal{A}_{n}$ is linear. Here is one simple observation.

\bigskip

\begin{lemma}
Suppose $A\in\mathcal{M}_{n}\left(  \mathbb{C}\right)  $ and $A\geq0$. Then
$p_{A}\in\mathcal{P}_{n}$.
\end{lemma}

\begin{proof}
Every nonnegative $A$ is a sum of rank-one nonnegative matrices. It follows
from the remarks preceding this lemma that we can assume that $rankA=1,$ which
means that $A$ can be written as
\[
A=\left(  \beta_{1},\ldots,\beta_{n}\right)  \left(  \beta_{1},\ldots
,\beta_{n}\right)  ^{\ast}=\left(  \beta_{i}\bar{\beta}_{j}\right)  .
\]
Then
\[
p_{A}=\sum_{i,j=1}^{n}\beta_{i}\bar{\beta}_{j}u_{i}^{\ast}u_{j}=\left(
\sum_{i=1}^{n}\bar{\beta}_{i}u_{i}\right)  ^{\ast}\left(  \sum_{i=1}^{n}%
\bar{\beta}_{i}u_{i}\right)  \in\mathcal{P}_{n}.
\]

\end{proof}

\begin{remark}
One more observation is that $p_{S}=p_{T}$ if and only if $S-T\in
\mathcal{D}_{n,0}$. The the map $T\mapsto p_{T}$ is really a function on
$\mathcal{M}_{n}\left(  \mathbb{C}\right)  /\mathcal{D}_{n,0}$. Thus if $A$ is
the sum of a positive semidefinite matrix and a zero-trace diagonal matrix,
then $p_{A}\in\mathcal{P}_{n}$. This means that Problem 1 is related to this scenario.
\end{remark}

\bigskip

An affirmative answer to Problem 1 yields an affirmative answer to the
following problem. If Connes' embedding problem has an affirmative answer,
then so must the following problem. Since the set $\mathcal{F}_{n}$ of
correlation matrices has such a simple definition, this question should be
easier to resolve.

\textbf{Problem 3. }If $W_{c}\left(  A\right)  \subset\left(  0,\infty\right)
$, the must $p_{A}\in\mathcal{P}_{n}$?

\section{The Case $n=2$.}

All of the questions can be answered when $n=2$. A complete description of
$W\left(  T\right)  $ when $T$ is a $2\times2$ matrix is given in \cite{PRH}.

\begin{lemma}
\label{2x2}If $T=\left(
\begin{array}
[c]{cc}%
a & b\\
c & d
\end{array}
\right)  ,$ then%
\[
W_{c}\left(  T\right)  =W_{uc}\left(  T\right)  =\frac{a+d}{2}+W\left(
\left(
\begin{array}
[c]{cc}%
0 & b\\
c & 0
\end{array}
\right)  \right)  =W\left(  \left(
\begin{array}
[c]{cc}%
\frac{a+d}{2} & b\\
c & \frac{a+d}{2}%
\end{array}
\right)  \right)  ,
\]
and $W\left(  \left(
\begin{array}
[c]{cc}%
0 & b\\
c & 0
\end{array}
\right)  \right)  $ is

\begin{enumerate}
\item the disk centered at $0$ with radius $\frac{1}{2}\max\left(  \left\vert
b\right\vert ,\left\vert c\right\vert \right)  $ if $bc=0$.

\item the segment from $-\sqrt{bc}$ to $\sqrt{bc}$ if $\left\vert b\right\vert
=\left\vert c\right\vert ,$ and

\item the elliptical disk with foci $\pm\sqrt{bc}$ if $bc\neq0$ and
$\left\vert b\right\vert \neq\left\vert c\right\vert .$
\end{enumerate}
\end{lemma}

\bigskip

The following result give the answers to most of the questions in the
preceding question when $n=2$.

\bigskip

\begin{corollary}
Suppose $T=\left(
\begin{array}
[c]{cc}%
a & b\\
c & d
\end{array}
\right)  \in\mathcal{M}_{2}\left(  \mathbb{C}\right)  $. If $W_{uc}\left(
T\right)  =W_{c}\left(  T\right)  \subseteq\lbrack0,\infty)$, then $T$ is the
sum of a trace-zero diagonal matrix $\left(
\begin{array}
[c]{cc}%
a & 0\\
0 & d
\end{array}
\right)  -\tau_{2}\left(  T\right)  I_{2}$ and a positive semidefinite matrix
$\left(
\begin{array}
[c]{cc}%
\frac{a+d}{2} & b\\
c & \frac{a+d}{2}%
\end{array}
\right)  $. So $p_{T}\in\mathcal{P}_{2}$.
\end{corollary}

\bigskip

A well known result of V. Pelligrini \cite{Pel} and a result of R. Kadison
\cite{Kad} implies that a linear numerical range-preserving map $\varphi
:\mathcal{M}_{n}\left(  \mathbb{C}\right)  \rightarrow\mathcal{M}_{n}\left(
\mathbb{C}\right)  $ has the form%
\[
\varphi\left(  S\right)  =U^{\ast}SU
\]
or
\[
\varphi\left(  S\right)  =U^{\ast}S^{T}U
\]
for some unitary matrix $U$. See \cite{LS} for more general results.

\bigskip

\begin{theorem}
Suppose $\varphi:\mathcal{M}_{2}\left(  \mathbb{C}\right)  \rightarrow
\mathcal{M}_{2}\left(  \mathbb{C}\right)  $ is linear. The following are equivalent.

\begin{enumerate}
\item For every $T\in\mathcal{M}_{2}\left(  \mathbb{C}\right)  $ we have
$W_{c}\left(  \varphi\left(  T\right)  \right)  =W_{c}\left(  T\right)  .$

\item There is a linear map $\alpha:\mathcal{M}_{2}\left(  \mathbb{C}\right)
\rightarrow\mathcal{D}_{2,0}$ and a unitary $U\in\mathcal{D}_{2}$ such that%
\[
\varphi\left(  T\right)  =U^{\ast}TU+\alpha\left(  T\right)
\]

\end{enumerate}
\end{theorem}

\begin{proof}
We know that if $T=T^{\ast}$, then $W_{c}\left(  \varphi\left(  T\right)
\right)  \subset\mathbb{R}$, which implies $\operatorname{Im}\varphi\left(
T\right)  \in\mathcal{D}_{2,0}$. More generally, this implies that, for every
$T\in\mathcal{M}_{2}\left(  \mathbb{C}\right)  $, we have%
\[
\varphi\left(  T^{\ast}\right)  -\varphi\left(  T\right)  ^{\ast}%
\in\mathcal{D}_{2,0}.
\]
Let $\left\{  e_{ij}:1\leq i,j\leq2\right\}  $ be the standard matrix units
for $\mathcal{M}_{2}\left(  \mathbb{C}\right)  $. We know that $W_{c}\left(
\varphi\left(  e_{12}\right)  \right)  =W\left(  e_{12}\right)  $ is the disk
centered at $0$ with radius $1/2.$ Hence, by Lemma \ref{2x2}, there is a
$\lambda\in\mathbb{C}$ with $\left\vert \lambda\right\vert =1$ and a
$D_{12}\in\mathcal{D}_{n,0}$ such that either $\varphi\left(  e_{12}\right)
=\lambda e_{12}+D_{12}$ or $\varphi\left(  e_{12}\right)  =\lambda
e_{21}+D_{12}$.

Case 1: $\varphi\left(  e_{12}\right)  =\lambda e_{12}$. It follows from the
fact that $e_{11}\in\mathcal{D}_{2,0}$ that
\[
W_{c}\left(  \varphi\left(  e_{11}\right)  -e_{11}\right)  =W_{c}\left(
\varphi\left(  e_{11}\right)  \right)  -\tau_{2}\left(  e_{11}\right)
=\left\{  0\right\}  .
\]
Whence, $D_{11}=\varphi\left(  e_{11}\right)  -e_{11}\in\mathcal{D}_{2,0}$.
Similarly, $D_{22}=\varphi\left(  e_{22}\right)  -e_{22}\in\mathcal{D}_{2,0}$.
Also
\[
D_{21}=\varphi\left(  e_{21}\right)  -\bar{\lambda}e_{21}=\varphi\left(
e_{12}^{\ast}\right)  -\varphi\left(  e_{12}\right)  ^{\ast}\in\mathcal{D}%
_{2,0}.
\]

Define $\alpha:\mathcal{M}_{2}\left(  \mathbb{C}\right)  \rightarrow
\mathcal{D}_{2,0}$ by
\[
\alpha\left(  \left(
\begin{array}
[c]{cc}%
a & b\\
c & d
\end{array}
\right)  \right)  =aD_{11}+bD_{12}+cD_{21}+dD_{22}.
\]
Hence
\[
\varphi\left(  \left(
\begin{array}
[c]{cc}%
a & b\\
c & d
\end{array}
\right)  \right)  =\left(
\begin{array}
[c]{cc}%
a & \lambda b\\
\bar{\lambda}c & d
\end{array}
\right)  +\alpha\left(  \left(
\begin{array}
[c]{cc}%
a & b\\
c & d
\end{array}
\right)  \right)  =
\]%
\[
\left(
\begin{array}
[c]{cc}%
\lambda & 0\\
0 & 1
\end{array}
\right)  \left(
\begin{array}
[c]{cc}%
a & b\\
c & d
\end{array}
\right)  \left(
\begin{array}
[c]{cc}%
\lambda & 0\\
0 & 1
\end{array}
\right)  +\alpha\left(  \left(
\begin{array}
[c]{cc}%
a & b\\
c & d
\end{array}
\right)  \right)  .
\]

Case 2: $\varphi\left(  e_{12}\right)  =\lambda e_{21}+D_{12}.$ If we define
$\psi\left(  A\right)  =\varphi\left(  A\right)  ^{T}$, then $W_{c}\left(
\psi\left(  T\right)  \right)  =W_{c}\left(  T\right)  $ always holds and
$\psi\left(  e_{12}\right)  =\lambda e_{12}+D_{12}^{T}.$ Hence, by Case 1,
there is a linear function $\alpha:\mathcal{M}_{2}\left(  \mathbb{C}\right)
\rightarrow\mathcal{D}_{2,0}$ such that
\[
\psi\left(  \left(
\begin{array}
[c]{cc}%
a & b\\
c & d
\end{array}
\right)  \right)  =\left(
\begin{array}
[c]{cc}%
\lambda & 0\\
0 & 1
\end{array}
\right)  \left(
\begin{array}
[c]{cc}%
a & b\\
c & d
\end{array}
\right)  \left(
\begin{array}
[c]{cc}%
\lambda & 0\\
0 & 1
\end{array}
\right)  +\alpha\left(  \left(
\begin{array}
[c]{cc}%
a & b\\
c & d
\end{array}
\right)  \right)  .
\]
Hence,
\[
\varphi\left(  \left(
\begin{array}
[c]{cc}%
a & b\\
c & d
\end{array}
\right)  \right)  =\psi\left(  \left(
\begin{array}
[c]{cc}%
a & b\\
c & d
\end{array}
\right)  \right)  ^{T}=
\]%
\[
\left(
\begin{array}
[c]{cc}%
0 & 1\\
\lambda & 0
\end{array}
\right)  \left(
\begin{array}
[c]{cc}%
a & b\\
c & d
\end{array}
\right)  \left(
\begin{array}
[c]{cc}%
0 & 1\\
\lambda & 0
\end{array}
\right)  ^{\ast}+\alpha\left(  \left(
\begin{array}
[c]{cc}%
a & b\\
c & d
\end{array}
\right)  \right)  ^{T}.
\]

\end{proof}

\bigskip

\noindent\textbf{Acknowledgement.} The second author is supported by a
research grant from the National Science Foundation.

\end{document}